\documentclass{amsart}
\usepackage{amssymb,hyperref}
\newtheorem{theorem}{Theorem}[section]

\newtheorem{definition}[theorem]{Definition}
\newtheorem{lemma}[theorem]{Lemma}
\newcommand{\Sym}{\mathop{\mathrm{Sym}}}
\newcommand{\Alt}{\mathop{\mathrm{Alt}}}
\newcommand{\Aut}{\mathop{\mathrm{Aut}}}
\newcommand{\PSL}{\mathop{\textrm{PSL}}}
\newcommand{\PGL}{\mathop{\mathrm{PGL}}}
\newcommand{\PGammaL}{\mathop{\mathrm{P}\Gamma\mathrm{L}}}

\newcommand{\Fix}{\mathop{\mathrm{Fix}}}
\newcommand{\GL}{\mathop{\mathrm{GL}}}
\newcommand{\Sz}{\mathop{\mathrm{Sz}}}

\newcommand{\AGL}{\mathop{\textrm{AGL}}}
\newcommand{\ASL}{\mathop{\textrm{ASL}}}
\newcommand{\soc}{\mathop{\textrm{soc}}}
\renewcommand{\wr}{\mathop{\textrm{wr}}}
\newcommand{\Sp}{\mathop{\textrm{Sp}}}

\newcommand{\fix}{\mathop{\mathrm{fix}}}
\newcommand{\orb}{\mathop{\mathrm{orb}}}
\def\norm#1#2{{\bf N}_{{#1}}{{(#2)}}}

\begin{document}

\title[Primitive Groups and hypergraphs]{Finite primitive groups and edge-transitive hypergraphs}

\author[P. Spiga]{Pablo Spiga}
\address{Pablo Spiga, Dipartimento di Matematica e Applicazioni,\newline
University of Milano-Bicocca,
 Via Cozzi 55 Milano, MI 20125, Italy} \email{pablo.spiga@unimib.it}

\thanks{Address correspondence to P. Spiga,
E-mail: pablo.spiga@unimib.it}

\begin{abstract}
We determine all finite primitive groups that are automorphism groups of edge-transitive hypergraphs. This gives an answer to a problem proposed by Babai and Cameron.\\
\begin{center}
\textit{Dedicated in the memory of \'{A}kos Seress}
\end{center}
\end{abstract}

\maketitle
\section{Introduction}\label{introduction}

A \textit{hypergraph} is a pair $\mathcal{H}=(\Omega,\mathcal{E})$, where $\Omega$ is a set and $\mathcal{E}$ is a set of subsets of $\Omega$. The elements of $\Omega$ are called \textit{vertices} and the elements of $\mathcal{E}$ are called \textit{edges}. The hypergraph $\mathcal{H}$ is called $r$-\textit{uniform} if $|E|=r$ for every $E\in \mathcal{E}$, and \textit{uniform} if $\mathcal{H}$ is $r$-uniform for some $r$. (Clearly,  $2$-uniform hypergraphs are the usual simple graphs.) Furthermore, $\mathcal{H}$ is \textit{edge-transitive} if the automorphism group $\Aut(\mathcal{H})$  of $\mathcal{H}$ acts transitively on the edges of $\mathcal{H}$; observe that every edge-transitive hypergraph is uniform.

Recently Laszlo Babai and Peter Cameron~\cite[Corollary~$1.2$]{BC} have shown that, for sufficiently large  sets $\Omega$, every finite primitive group $G$ on $\Omega$ with $G\neq\Alt(\Omega)$  is the automorphism group of an edge-transitive hypergraph, that is, $G=\Aut(\mathcal{H})$ for some edge-transitive hypergraph $\mathcal{H}=(\Omega,\mathcal{E})$. Considering that not every transitive group  is the automorphism group of a graph and that rarely a primitive group is the automorphism group of a graph, in our opinion this result comes  with considerable surprise.  Observe that if $\Alt(\Omega)\leq \Aut(\mathcal{H})$, then $\Sym(\Omega)=\Aut(\mathcal{H})$ and hence  $\Alt(\Omega)$ is not the automorphism group of any hypergraph (let alone edge-transitive).

In this paper we refine~\cite[Corollary~$1.2$]{BC} and we obtain the explicit list of finite primitive groups which are not automorphism groups of edge-transitive hypergraphs.
\begin{theorem}\label{thrm:main}
Let $G$ be a finite primitive group on $\Omega$ with $G\neq \Alt(\Omega)$. Then either there exists an edge-transitive hypergraph $\mathcal{H}=(\Omega,\mathcal{E})$ with $G=\Aut(\mathcal{H})$, or $G$ is one of the groups in Table~$\ref{table1}$.
\end{theorem}
\begin{table}[!h]
\begin{tabular}{|l|l|}\hline
Deg. $5$&$C_5$, $\AGL_1(5)$\\
Deg. $6$&$\PGL_2(5)$\\
Deg. $7$&$C_7$, $C_7\rtimes C_3$\\
Deg. $8$&$\AGL_1(8)$, $\mathrm{A}\Gamma\mathrm{L}_1(8)$,  $\PSL_2(7)$\\
Deg. $9$&$(C_3\times C_3)\rtimes C_4$, $\AGL_1(9)$, $(C_3\times C_3)\rtimes Q_8$, $\ASL_2(3)$, $\PSL_2(8)$, $\mathrm{P}\Gamma\mathrm{L}_2(8)$\\
Deg. $10$&$\PSL_2(9)$, $\PGL_2(9)$\\\hline
\end{tabular}
\caption{Primitive groups that are not automorphism groups of edge-transitive hypergraphs}\label{table1}
\end{table}
Yet again, we find rather surprising that the list of exceptions is so short. The groups $\AGL_1(5)$, $\PGL_2(5)$, $\PSL_2(8)$ and $\mathrm{P}\Gamma\mathrm{L}_2(8)$ are set-transitive and hence are genuine exceptions in Theorem~\ref{thrm:main}: like the alternating group $\Alt(\Omega)$, a proper set-transitive subgroup of $\Sym(\Omega)$ cannot be the automorphism group of any family of subsets of $\Omega$. It is easy to check (for example with the computer algebra system \texttt{magma}~\cite{magma}) that all the groups in Table~$1$ are genuine exceptions. 

There is a tight analogue between~\cite[Corollary~$1.2$]{BC} and Theorem~\ref{thrm:main} and two well-known results in the literature. In fact, Cameron, Neumann and Saxl~\cite{CNS} have shown that, apart the alternating and the symmetric group,  every finite primitive group on $\Omega$, with $|\Omega|$ sufficiently large,  has a regular orbit on the set of subsets of $\Omega$. Later,~Seress~\cite{Seress} has computed the explicit list of exceptions.  Here the analogy between~\cite{CNS} and~\cite{BC} and between~\cite{Seress} and Thereom~\ref{thrm:main}  is not purely aesthetic: some probabilistic arguments in~\cite{BC,CNS} have a very similar  flavour, and we use both the main result and some key ideas in~\cite{Seress} to prove Theorem~\ref{thrm:main}.

We observe that (with different terminology)~\cite[Theorem~$4.2$]{DVS} shows that, apart the alternating groups and ten explicit exceptions, every finite primitive group  on $\Omega$ is the automorphism group of a hypergraph $\mathcal{H}=(\Omega,\mathcal{E})$. The hypergraphs considered in~\cite{DVS} are rather far from being uniform (let alone being edge-transitive) and hence Theorem~\ref{thrm:main} improves~\cite[Theorem~$4.2$]{DVS}.

Our proof of Theorem~\ref{thrm:main} requires a detailed knowledge of the structure of the finite primitive groups and in particular we use the O'Nan-Scott theorem combined with the Classification of the Finite Simple Groups.

\subsection*{Acknowledgements}I am in debt with Peter Cameron for suggesting this problem and with  the organisers of the conference: ``New trends in algebraic combinatorics'', in Villanova, in June 2014. This wonderful environment was extremely fruitful and gave me the opportunity  to discover and  discuss~\cite{BC} with Peter. 

I am also in debt with Primo\v{z} Poto\v{c}nik for hosting the heavy computer computations required in the proof of Theorem~\ref{thrm:main}.

\subsection{Computer computations}
All the computations in this paper are done with the computer algebra system \texttt{magma}~\cite{magma}. These computations require a considerable amount of patience but can be performed with standard built-in \texttt{magma} functions. 

Given a primitive group $G$ on $\Omega$, we use a ``random'' approach to exhibit an edge-transitive  hypergraph $\mathcal{H}=(\Omega,\mathcal{E})$ with $G=\Aut(\mathcal{H})$, that is, we generate a random subset $\Delta$ of $\Omega$ of small cardinality ($|\Delta|\leq 6$) and we check whether $G=\Aut(\Omega,\mathcal{E})$ with $\mathcal{E}=\Delta^G=\{\Delta^g\mid g\in G\}$. Except for the groups in Table~$1$ (which are not automorphism groups of edge-transitive hypergraphs), typically with this method we succeed with no more than three trials. In particular, this suggests that, for most primitive groups, the proportion of subsets $\Delta$ of $\Omega$ with $G=\Aut(\Omega,\Delta^G)$ is very large. 

Only a handful of cases required a thorough  analysis. For instance, $\PSL_3(4)$ in its action on the $21$ points of the projective plane of order four is the automorphism group of an edge-transitive $10$-uniform hypergraph,  but is not the automorphism group of any edge-transitive $r$-uniform hypergraph for $r\notin\{ 10,11\}$.

\subsection{Notation}\label{notation}
Let $\Omega$ be a finite set and let $g$ be a permutation on $\Omega$. We denote by $\Fix_\Omega(g)$ the set $\{\omega\in \Omega\mid \omega^g=\omega\}$, by $\fix_\Omega(g)$ the cardinality $|\Fix_\Omega(g)|$ and  by $\orb_\Omega(g)$  the number of cycles of $g$   (in its decomposition in disjoint cycles). Similarly, if $C$ is the cyclic group generated by $g$, we write $\Fix_\Omega(C)=\Fix_\Omega(g)$, $\fix_\Omega(C)=\fix_\Omega(g)$ and $\orb_\Omega(C)=\orb_\Omega(g)$.

Given a group $G$, we denote by $\mathcal{C}(G)$ the set of subgroups of prime order of $G$.

Let $G$ be a primitive group on $\Omega$ with $\Alt(\Omega)\nleq G$. One of the main ingredients in the proof of Theorem~\ref{thrm:main} is the structure of the lattice  of  overgroups of $G$: here we will be using the results obtained by Aschbacher in~\cite{Asch1,Asch} and  by Liebeck, Praeger and Saxl in~\cite{LPS1,Pra}. Following the notation in~\cite{Asch1,Asch}, we denote by $\mathcal{O}(G)$ the lattice $\{M\leq \Sym(\Omega)\mid G\leq M\}$ (ordered by set inclusion) and by $\mathcal{M}(G)$ the maximal elements of $\mathcal{O}(G)\setminus\{\Alt(\Omega),\Sym(\Omega)\}$. 

With a slight abuse of terminology and following~\cite{BC},  we say that a subgroup $M$ of $\Sym(\Omega)$ is \textit{maximal} if $M\notin\{\Alt(\Omega) ,\Sym(\Omega)\}$ and either $M$ is a maximal subgroup of $\Sym(\Omega)$, or  $\Alt(\Omega)$ is the only proper subgroup of $\Sym(\Omega)$ containing $M$. In particular, the elements of $\mathcal{M}(G)$ are exactly the maximal subgroups containing $G$. From~\cite{LPS1}, we see that every maximal subgroup has O'Nan-Scott type HA (``holomorphic abelian''), AS (``almost simple''), PA (``product action'')  or SD (``simple diagonal''). Thus, according to this subdivision, we partition the elements of $\mathcal{M}(G)$ in four pair-wise disjoint sets $\mathcal{M}_{\textrm{HA}}(G)$, $\mathcal{M}_{\textrm{AS}}(G)$, $\mathcal{M}_{\textrm{PA}}(G)$ and $\mathcal{M}_{\textrm{SD}}(G)$.

We denote by $\soc(G)$ the socle of $G$. Given a subset $\Delta$ of $\Omega$, we denote by $G_\Delta$ the set-wise stabiliser $\{g\in G\mid \Delta^g=\Delta\}$. We denote by $2^\Omega$ the power-set of $\Omega$, that is, the set of subsets of $\Omega$. Moreover, we define 
\begin{align}\label{alin}
\mathcal{F}(M)&=\{\Delta\in 2^\Omega\mid \Delta^g=\Delta \textrm{ for some }g\in M\setminus\{1\}\}, \textrm{ and}\\\nonumber
\mathcal{S}(G)&=\bigcup_{M\in\mathcal{M}(G)}\mathcal{F}(M).
\end{align}
Thus $\mathcal{F}(M)$ consists of the subsets of $\Omega$ fixed by some non-identity element of $M$ and similarly $\mathcal{S}(G)$ consists of the subsets of $\Omega$ fixed by some non-identity element of some maximal subgroup containing $G$.

In this paper, we use the subdivision of the finite primitive groups in eight types as suggested by Laszlo Kov\'{a}cs, and then formulated by Praeger, Liebeck and Saxl (see~\cite{LPS2}, or~\cite[Section~$3$]{Pra} which has a formulation closer to our application).


\section{Basic lemmas}\label{basiclemma}
We start with two elementary observations, which are the backbone  underlying the idea in the proof of Theorem~\ref{thrm:main}.

\begin{lemma}\label{basic-2}
Let $G$ be a finite primitive group on $\Omega$ with $\Alt(\Omega)\nleq G$. If $\mathcal{S}(G)\subsetneq 2^\Omega$, then  there exists an edge-transitive hypergraph $\mathcal{H}=(\Omega,\mathcal{E})$ with $G=\Aut(\mathcal{H})$.
\end{lemma}
\begin{proof}
Let $\Delta$ be an element of $2^\Omega$ with $\Delta\notin \mathcal{S}(G)$. Replacing $\Delta$ by $\Omega\setminus \Delta$ if necessary, we may assume that $1\leq |\Delta|\leq |\Omega|/2$. Set $\mathcal{E}=\Delta^G=\{\Delta^g\mid g\in G\}$, $\mathcal{H}=(\Omega,\mathcal{E})$ and $A=\Aut(\mathcal{H})$. By construction $\mathcal{H}$ is an edge-transitive hypergraph and $G\leq A$. Write $n=|\Omega|$ and $m=|\Delta|$.

Suppose that $\Alt(\Omega)\nleq A$. Then $A\in\mathcal{O}(G)\setminus\{\Alt(\Omega),\Sym(\Omega)\}$ and hence there exists $M\in\mathcal{M}(G)$ with $G\leq A\leq M$. Since $\Delta\notin\mathcal{F}(M)$, we get $M_\Delta=1$ and hence $A_\Delta=1$. It follows that
$$|A|=|A:A_\Delta|=|\Delta^A|=|\mathcal{E}|=|\Delta^G|=|G:G_\Delta|=|G|$$
and thus $A=G$.

Suppose that $\Alt(\Omega)\leq A$. (We show that this case cannot occur.) Then $A=\Sym(\Omega)$, $\mathcal{E}=\{\Lambda\in 2^\Omega\mid |\Lambda|=m\}$ and $G$ is transitive on the $m$-subsets of $\Omega$. As $\Delta\notin\mathcal{S}(G)$, for every $M\in \mathcal{M}(G)$, we get $G_\Delta=M_\Delta=1$ and hence $|G|={n\choose m}=|M|$. It follows that $\mathcal{M}(G)=\{G\}$ (that is, $G$ is a maximal subgroup of $\Sym(\Omega)$) and $G$ acts regularly on the subsets of $\Omega$ of cardinality $m$. Assume that $m=1$. Then $G$ acts regularly on $\Omega$ and, by primitivity, has prime order. This contradicts the fact that $G$ is maximal. Thus $m\geq 2$. Therefore $|G|={n\choose m}<n!/(n-m)!=n(n-1)\cdots (n-m+1)$ and hence $G$ is not $m$-transitive. The main result of~\cite{Kantor} (see also~\cite[Theorem~$9.4$B]{DM} for the notation) gives that one of the following happens:
\begin{description}
\item[(i)]$m=2$, $\ASL_1(q)\leq G\leq \mathrm{A}\Sigma\mathrm{L}_1(q)$ with $q\equiv 3\mod 4$;
\item[(ii)]$m=3$, $\PSL_2(q)\leq G\leq \mathrm{P}\Sigma\mathrm{L}_2(q)$ with $q\equiv 3\mod 4$;
\item[(iii)]$m=3$,  $G\in \{\AGL_1(8),\mathrm{A}\Gamma\mathrm{L}_1(8),\mathrm{A}\Gamma\mathrm{L}_1(32)\}$;
\item[(iv)]$m=4$, $G\in \{\PGL_2(8),\mathrm{P}\Gamma\mathrm{L}_2(8),\mathrm{P}\Gamma\mathrm{L}_2(32)\}$.
\end{description}
A quick inspection  reveals that the groups in this list are either not maximal or do not act regularly on the $m$-subsets of $\Omega$. This final contradiction concludes the proof.
\end{proof}

\begin{lemma}\label{basic}
Let $G$ be a finite primitive group on $\Omega$ with $\Alt(\Omega)\nleq G$. Suppose that $|\mathcal{M}(G)|=1$. Then either there exists an edge-transitive hypergraph $\mathcal{H}=(\Omega,\mathcal{E})$ with $G=\Aut(\mathcal{H})$, or $G$ is one of the groups in Table~$\ref{table1}$.
\end{lemma}
\begin{proof}
Let $M$ be the maximal subgroup with $\mathcal{M}(G)=\{M\}$. Then $\mathcal{S}(G)=\mathcal{F}(M)$. If $\mathcal{F}(M)\subsetneq 2^\Omega$, then the proof follows from Lemma~\ref{basic-2}. Suppose that $\mathcal{F}(M)=2^\Omega$, that is, $M$ has no regular orbit on the set of subsets of $\Omega$. Then $M$ is one of the forty-three groups given in~\cite[Theorem~$2$]{Seress} and in particular $5\leq |\Omega|\leq 17$, or $21\leq |\Omega|\leq 24$, or $|\Omega|=32$. Now the proof follows with a case-by-case analysis  using the library of small  primitive groups in the computer algebra system \texttt{magma}.
\end{proof}


\section{Primitive groups of HS and SD type}\label{SD}
In this section we prove Theorem~\ref{thrm:main} when $G$ is a primitive group of HS or SD type. This is by far the easiest case to deal with.
\begin{theorem}\label{typeHS-SD}
Let $G$ be a finite primitive group on $\Omega$ of HS or SD type. Then there exists an edge-transitive hypergraph $\mathcal{H}=(\Omega,\mathcal{E})$ with $G=\Aut(\mathcal{H})$.
\end{theorem}
\begin{proof}
Let $M\in \mathcal{M}(G)$. From~\cite[Proposition~$8.1$]{Pra}, we have $\soc(M)=\soc(G)$ and hence $M\leq \norm{\Sym(\Omega)}{\soc (G)}$. By the maximality of $M$, we have $M=\norm {\Sym(\Omega)}{\soc (G)}$. This shows that $\norm {\Sym(\Omega)}{\soc (G)}$ is the unique maximal subgroup of $\Sym(\Omega)$ containing $G$ and $|\mathcal{M}(G)|=1$. Now the proof follows from Lemma~\ref{basic}.
\end{proof}

Before dealing with other O'Nan-Scott types, we highlight the main ingredients in the proof of Theorem~\ref{typeHS-SD}. First, it is necessary to have a detailed knowledge of all  maximal overgroups of $G$. The work in~\cite{Asch1,Asch} and in~\cite{LPS1,Pra} deals with the inclusion problem among primitive groups and is fundamental for our application. Second, it is necessary to establish the existence of a subset $\Delta$ of $\Omega$ with $M_\Delta=1$, for every $M\in \mathcal{M}(G)$. Rarely we will be able  (as in the proof above) to simply invoke~\cite[Theorem~$2$]{Seress}. However, a probabilistic approach (which is also one of the fundamental tools in~\cite{BC,CNS,Seress}) will often reduce this second problem to the case that $G$ has small degree.


\section{Some more basic lemmas and some estimates}\label{estimates}
We following four facts can be hardly called lemmas, but they will prove useful.
\begin{lemma}\label{basic0}
Let $L$ be a transitive group of degree $\ell$. Then $L$ has at most $\ell^{\log_2(\ell)}$ systems of imprimitivity. 
\end{lemma}
\begin{proof}
Let $\Delta$ be the set acted upon by $L$ and let $\delta\in \Delta$. The systems of imprimitivity of $L$ are in one-to-one correspondence with the subgroups of $L$ containing $L_\delta$.  Every subgroup $U$ of $L$ with $L_\delta\leq U$ is generated by $L_\delta$ and by some right cosets of $L_\delta$, that is, $U=\langle L_\delta, L_\delta x_1,\ldots,L_\delta x_v\rangle$ for some right cosets  $L_\delta x_1,\ldots,L_\delta x_v$ of $L$. As $|L:L_\delta|=|\Delta|=\ell$, we may choose $v\leq \log_2(\ell)$. It follows that there are at most $\ell^{\log_2(\ell)}$ choices for $U$.
\end{proof}

\begin{lemma}\label{basic3}
Let $g$ be a permutation of $\Omega$ and let $p$ the smallest prime dividing the order of $g$. Then $\orb_\Omega (g)\leq (|\Omega|+(p-1)\fix_\Omega(g))/p$ and $2^{\orb_\Omega(g)}=|\{\Delta\in 2^\Omega\mid \Delta^g=\Delta\}|$. In particular, $|\{\Delta\in 2^\Omega\mid \Delta^g=\Delta\}|\leq 2^{(|\Omega|+(p-1)\fix_\Omega(g))/p}$.
\end{lemma}
\begin{proof}
The element $g$ has cycles of size $1$ on $\Fix_\Omega(g)$ and of size at least $p$ on $\Omega\setminus \Fix_\Omega(g)$. Thus $\orb_\Omega(g)\leq |\Fix_\Omega(g)|+|\Omega\setminus\Fix_\Omega(g)|/p=(|\Omega|+(p-1)\fix_\Omega(g))/p$. The rest of the lemma is obvious.
\end{proof}


\begin{lemma}\label{fact3}Let $M$ be a permutation group on $\Omega$. Then $$|\mathcal{F}(M)|\leq \sum_{C\in \mathcal{C}(M)}2^{\orb_\Omega(C)}\leq \sum_{C\in \mathcal{C}(M)}2^{\frac{|\Omega|}{2}+\frac{\fix_\Omega(C)}{2}}.$$
\end{lemma}
\begin{proof}
As $\mathcal{F}(M)=\bigcup_{C\in\mathcal{C}(M)}\{\Delta\in 2^\Omega\mid \Delta^C=\Delta\}$, the proof follows from Lemma~\ref{basic3}.
\end{proof}

Let $G$ be a primitive group on $\Omega$ and let $M\in \mathcal{M}_{\textrm{PA}}(G)$ with $M\cong \Sym(m)\wr\Sym(\ell)$. Then the set $\Omega$ admits an $M$-invariant Cartesian decomposition $\Omega=\Delta_1\times \cdots\times \Delta_\ell$ with $|\Delta_i|=m\geq 5$ and $\ell\geq 2$. As $G\leq M$, the Cartesian decomposition $\Omega=\Delta_1\times \cdots \times \Delta_\ell$ is  also $G$-invariant. Conversely, if $\Omega=\Delta_1\times \cdots \times \Delta_\ell$ is a $G$-invariant Cartensian decomposition with $|\Delta_i|\geq 5$ and $\ell\geq 2$ then the permutation group $(\Sym(\Delta_1)\times \cdots \times\Sym(\Delta_\ell))\wr\Sym(\ell)\cong \Sym(|\Delta_1|)\wr\Sym(\ell)$  contains $G$ and is maximal (see~\cite[Theorem]{LPS1}), and hence it belongs to $\mathcal{M}_{\mathrm{PA}}(G)$. Thus we have shown the following.
\begin{lemma}\label{CDlemma}Let $G$ be a primitive group on $\Omega$. The elements of $\mathcal{M}_{\mathrm{PA}}(G)$ are in one-to-one correspondence with the $G$-invariant Cartesian decompositions $\Omega=\Delta_1\times \cdots\times \Delta_\ell$ with $|\Delta_i|\geq 5$ and $\ell\geq 2$.
\end{lemma}

Given a positive integer $m$ and $n=m^2$, we define 
\begin{eqnarray}
F_0(n)&=&2^{\frac{n}{2}+\frac{\sqrt{n}}{2}+(\sqrt{n}-1)\log_2(\sqrt{n})},\label{eq:def}\\
F'(n)&=&2^{n-\frac{11}{2}\sqrt{n}+8(\log_2(\sqrt{n}))^2-4\log_2(\sqrt{n})},\label{eq:def1}\\
F''(n)&=&2^{n+2\sqrt{n}\log_2(0.3967)+\sqrt{n}}.\label{eq:def2}
\end{eqnarray}
Moreover, for every prime number $p$ and for every non-negative integers $i$ and $j$ with $0<p(i+j)<11$, we define
\begin{equation}\label{eq:def3}
F_{i,j}^p(n)=\frac{2^{n-(i+j)(p-1)\sqrt{n}+ijp(p-1)}\sqrt{n}!^2}{(p-1)(\sqrt{n}-pi)!(\sqrt{n}-pj)!p^{i+j}i!j!}.
\end{equation}
Finally, set
\begin{eqnarray}\label{eq000}
F(n)=F_0(n)+F'(n)+F''(n)+\sum_{\substack{p \textrm{ prime},i,j\geq 0\\0<p(i+j)<11}}F_{i,j}^p(n).
\end{eqnarray}

Observe that each $F_{i,j}^p(n)$ can be written as an elementary function in $\sqrt{n}$. For example, when $(p,i,j)=(2,2,1)$, we have
$F_{2,1}^2(n)=2^{n-3\sqrt{n}+4}{\sqrt{n}^2(\sqrt{n}-1)^2(\sqrt{n}-2)(\sqrt{n}-3)}/16$. In particular, since we have only a handful of triples $(p,i,j)$ with $0<p(i+j)<11$, the function $F(n)$ is relatively easy and can be efficiently implemented in a computer.

\begin{lemma}\label{basic2}Let $m$ and $\ell$ be positive integers with $m\geq 5$ and $\ell\geq 2$, and let $n=m^\ell$. Let $M$ be the wreath product $\Sym(m)\wr\Sym(\ell)$ endowed of its natural product action on 
$\Omega=\Delta^\ell$ with $\Delta=\{1,\ldots,m\}$. If $\ell\geq 3$, or if $\ell=2$ and $m\geq 36$, then
$|\mathcal{F}(M)|\leq 
F(n)$.
(The function $F(n)$ is defined in Eq.~$\eqref{eq000}$.)
\end{lemma}
\begin{proof}

Let $B=\Sym(\Delta)^\ell$ be the base group of $M$. We denote the elements of $M$ by $(h_1,\ldots,h_\ell)\sigma$, with $\sigma\in \Sym(\ell)$ and $h_1,\ldots,h_\ell\in \Sym(\Delta)$.

Let $g=(h_1,\ldots,h_\ell)\sigma\in M$, with $\sigma\neq 1$. We claim that $\fix_\Omega(g)\leq m^{\ell-1}$. Relabelling the index set $\{1,\ldots,\ell\}$ if necessary, we may assume that
$(1,\ldots, k)$ is a non-identity cycle of $\sigma$.  Let $\omega=(\delta_1,\ldots,\delta_\ell)\in \Omega$.  Now, 
\[
\omega^g=(\delta_k^{h_k},\delta_1^{h_1},\cdots,\delta_{k-2}^{h_{k-2}},\delta_{k-1}^{h_{k-1}},\delta_{k+1}',\ldots,\delta_{\ell}')
\]
for some $\delta_{k+1}',\ldots,\delta_{\ell}'\in\Delta$. In particular, if $\omega^g=\omega$, then 
$\delta_1=\delta_k^{h_k},\,\delta_2=\delta_1^{h_1},\ldots,\,\delta_{k}=\delta_{k-1}^{h_{k-1}}$, that is,
\[
\delta_k=\delta_1^{(h_k)^{-1}},\,\delta_{k-1}=\delta_1^{(h_{k-1}h_k)^{-1}},\,\ldots,\,\delta_2=\delta_1^{(h_2\cdots h_{k-1}h_k)^{-1}}.
\]
From this we deduce that the first $k$ coordinates of $\omega$ are uniquely determined  by the first coordinate $\delta_1$ of $\omega$. Therefore $\fix_\Omega(g)\leq m^{\ell-1}$.

Let now $g=(h_1,\ldots,h_\ell)\in B$ with $g\neq 1$. An easy computation shows that $\fix_\Omega(g)=\prod_{i=1}^\ell\fix_{\Delta}(h_i)\leq (m-2)m^{\ell-1}$.  Thus, as $m\geq 5$, we have $(m-2)m^{\ell-1}>m^{\ell-1}$ and hence
 $\fix_\Omega(g)\leq (m-2)m^{\ell-1}$, for every $g\in M$ with $g\neq 1$. Therefore Lemma~\ref{fact3} gives
\[
|\mathcal{F}(M)|\leq \sum_{C\in \mathcal{C}(M)}2^{n-m^{\ell-1}}<2^{n-m^{\ell-1}}|M|=2^{n-m^{\ell-1}}m!^\ell\ell!.
\]
Assume that $\ell\geq 3$. Using  Stirling's formula, with a computation, we get $m!\leq (0.5211\cdot m)^m$  for $m\geq 5$, and hence
\begin{eqnarray*}
m!^\ell\ell!&\leq &(0.5211\cdot m)^{m\ell}\ell^{\ell-1}=2^{m(\log_2(n)+\ell\log_2(0.5211))+(\ell-1)\log_2(\ell)}\\
&\leq& 2^{n^{1/3}(\log_2(n)+3\log_2(0.5211))+2\log_2(3)},
\end{eqnarray*}
where the last inequality follows  with a computation. Another computation gives
$$2^{n-n^{2/3}+n^{1/3}(\log_2(n)+3\log_2(0.5211))+2\log_2(3)}<F(n),$$
and hence this concludes the proof when $\ell\geq 3$.

Assume that $\ell=2$ and $m\geq 36$. Fix $c$ with $0\leq c\leq m-11$. We use Lemma~\ref{fact3} to bound $|\mathcal{F}(M)|$. Let $g=(h_1,h_2)\sigma\in M$ with $\sigma\neq 1$ and with $|g|$ prime. Then $|g|=2$, $\sigma=(1,2)$ and $1=g^2=(h_1h_2,h_2h_1)$,  hence $h_2=h_1^{-1}$ and $g=(h_1,h_1^{-1})\sigma$. Now, an easy computation shows that $\Fix_\Omega(g)=\{(\delta,\delta^{h_1})\mid \delta\in \Delta\}$, hence $\fix_\Omega(g)= m$ and $2^{\orb_\Omega(g)}=2^{n/2+\sqrt{n}/2}$. Since $m!=\sqrt{n}!\leq \sqrt{n}^{\sqrt{n}-1}=2^{(\sqrt{n}-1)\log_2(\sqrt{n})}$, we get $$\sum_{C\in\mathcal{C}(M),C\nleq B}2^{\orb_\Omega(g)}\leq 2^{\frac{n}{2}+\frac{\sqrt{n}}{2}+(\sqrt{n}-1)\log_2(\sqrt{n})}\overset{(\textrm{see }\eqref{eq:def})}{=}F_0(n).$$

We now focus on the subgroups of prime order of $B$. Given $h\in \Sym(\Delta)$ with $|h|=p$ a prime number, we say that $h$ is of type $p^i$ if $h$ is the product of $i$ cycles of length $p$. For each prime number $p$ and non-negative integers $i,j$ with $0<p(i+j)<11$, define
\begin{align*}
&\mathcal{C}_{i,j}^p&&=\{\langle(h_1,h_2)\rangle\in \mathcal{C}(M)\mid h_1\textrm{ of type }p^i \textrm{ and }h_2\textrm{ of type }  p^j\textrm{ on }\Delta\},\\
&\mathcal{C}'&&=\{\langle(h_1,h_2)\rangle\in\mathcal{C}(M)\mid\langle(h_1,h_2)\rangle\notin\mathcal{C}_{i,j}^p \textrm{ for every }p,i,j,\, \mathrm{fix}_\Delta(h_1)\geq c \textrm { and }\mathrm{fix}_\Delta(h_2)\geq c\},\\
&\mathcal{C}''&&=\{\langle(h_1,h_2)\rangle\in\mathcal{C}(M)\mid \mathrm{fix}_\Delta(h_1)<c \textrm{ or }\mathrm{fix}_\Delta(h_2)<c\}.
\end{align*}
Observe that $\mathcal{C}'$ is disjoint from $\mathcal{C}''$. Moreover, for every $\langle (h_1,h_2)\rangle\in \mathcal{C}_{i,j}^p$, we have $\fix_\Omega(h_1)=m-pi\geq m-10$ and $\fix_\Omega(h_2)=m-pj\geq m-10$. As $c<m-10$, we get that $\mathcal{C}_{i,j}^p$ is disjoint from $\mathcal{C}''$. Thus the sets $\mathcal{C}_{i,j}^p,\mathcal{C}',\mathcal{C}''$ are pair-wise disjoint. Furthermore, every subgroup of prime order of $B$ lies in exactly one of $\mathcal{C}_{i,j}^p,\mathcal{C}',\mathcal{C}''$. 

Observe that $\Sym(m)$ contains $\frac{m!}{(m-pi)!p^ii!}$ elements of type $p^i$ and hence $$|\mathcal{C}_{i,j}^p|=\frac{1}{p-1}\cdot\frac{m!^2}{(m-pi)!(m-pj)!p^{i+j}i!j!}.$$
Given $C\in \mathcal{C}_{i,j}^p$, we have $\fix_\Omega(C)=(m-pi)(m-pj)$ and $C$ has orbits of size $p$ on $\Omega\setminus\Fix_\Omega(C)$. Thus 
\begin{eqnarray*}
\sum_{C\in \mathcal{C}_{i,j}^p}2^{\orb_\Omega(C)}&=&
\frac{2^{(m-pi)(m-pj)+\frac{n-(m-pi)(m-pj)}{p}}m!^2}{(p-1)(m-pi)!(m-pj)!p^{i+j}i!j!}\\
&=&\frac{2^{n-(i+j)(p-1)m+ij(p^2-p)}m!^2}{(p-1)(m-pi)!(m-pj)!p^{i+j}i!j!}\overset{(\textrm{see }\eqref{eq:def3})}{=}F_{i,j}^p(n)
.
\end{eqnarray*} 

Set $\varepsilon=0.3967$.  Using  Stirling's formula we get $m!\leq (\varepsilon m)^m$  for $m\geq 36$.
Clearly,
$$|\mathcal{C}'|\leq \left({m\choose c}(m-c)!\right)^2=\left(\frac{m!}{c!}\right)^2< m^{2(m-c)}=2^{2(m-c)\log_2(m)}$$
and 
$$|\mathcal{C}''|<m!^2<(\varepsilon m)^{2m}=2^{2m(\log_2(m)+\log_2(\varepsilon))}.$$

Let $C$ be a subgroup of $B$ of prime order $r$ with $C\notin \bigcup_{i,j,p}\mathcal{C}_{i,j}^p$, and write $C=\langle g\rangle$ with $g=(h_1,h_2)$.  We claim that if $C\in \mathcal{C}'$, then $\fix_\Omega(g)\leq m(m-11)$. We argue by contradiction and we assume that $\fix_\Omega(g)>m(m-11)$. Assume that $h_1$ has type $r^x$ and $h_2$ has type $r^y$. As $C$ is not in any of the sets $\mathcal{C}_{i,j}^p$, we have $r(x+y)\geq 11$. Observe that $\fix_\Omega(g)=\fix_\Delta(h_1)\fix_\Delta(h_2)=(m-rx)(m-ry)\leq m(m-rx)$ and hence $rx<11$. Similarly, $ry<11$.  In particular, we have only a handful of triples $(r,x,y)$ with $r(x+y)\geq 11$, $rx<11$ and $ry<11$. By studying each of these triples in turn and using $m\geq 36$,  we obtain that the inequality $(m-rx)(m-ry)>m(m-11)$ is never satisfied. We do not give the full argument here, but we simply deal with the case that $(r,x,y)=(2,3,3)$ (all the other cases are similar). Now, $(m-rx)(m-ry)=(m-6)^2$ and the inequality $(m-6)^2>m(m-11)$ hold true only if $36>m$, a contradiction.

From the previous paragraph it follows that $2^{\orb_\Omega(C)}\leq 2^{n-\frac{11}{2}m}$ when $C\in \mathcal{C}'$. Let $C$ be in $\mathcal{C}''$ and write $C=\langle g\rangle$ with $g=(h_1,h_2)$. Now, $\fix_\Omega(g)=\fix_\Delta(h_1)\fix_\Delta(h_2)\leq (c-1)m$. Thus $2^{\orb_\Omega(C)}\leq 2^{n/2+(c-1)m/2}$ when $C\in \mathcal{C}''$. Therefore
$$\sum_{C\in \mathcal{C}'\cup\mathcal{C}''}2^{\orb_\Omega(C)}\leq 2^{n-\frac{11}{2}m+2(m-c)\log_2(m)}+2^{\frac{n}{2}+\frac{(c-1)m}{2}+2m(\log_2(m)+\log_2(\varepsilon))}.$$
Set $c=\lfloor m-4\log_2(m)+3\rfloor$ and observe that $c\leq m-11$ for $m\geq 36$. Now,  with a careful computation we get
$$\sum_{C\in \mathcal{C}'\cup\mathcal{C}''}2^{\orb_\Omega(C)}\leq 2^{n-\frac{11}{2}m+8(\log_2(m))^2-4\log_2(m)}+2^{n+2m\log_2(\varepsilon)+m}\overset{(\textrm{see }\eqref{eq:def1},\,\eqref{eq:def2})}{=}F'(n)+F''(n).$$
Now the proof follows from the definition of $F(n)$.
\end{proof}

Our choice of $c$ in the proof of Lemma~\ref{basic2}  is not asymptotically best possible, however it is the formulation that best suits our application.


Let $m$  be a positive integer.  Set
\begin{align*}
G'(m)&=2^{{m\choose 2}-\frac{11}{2}m+33+2(\log_2(m))^2+\log_2(m)},\\
G''(m)&=2^{{m\choose 2}-\frac{1}{2}m+\log_2(0.4)m+(\log_2(m))^2+2\log_2(m)+\frac{3}{4}}.
\end{align*}
For every prime number $p$ and positive integer $i$ with $ip<11$, we set
\begin{equation}\label{eq:def6}
G_i^p(m)=
\begin{cases}
2^{{m\choose 2}-i(p-1)m+\frac{i^2p(p-1)}{2}+\frac{i(p-1)}{2}}\cdot\frac{m!}{(p-1)(m-pi)!p^ii!}&\textrm{if }p>2,\\
2^{{m\choose 2}-im+i^2+i}\cdot\frac{m!}{(m-2i)!2^ii!}&\textrm{if }p=2.
\end{cases}
\end{equation}
Finally, we define
\begin{equation}\label{eq0000}
G(m)=G'(m)+G''(m)+\sum_{\substack{i\geq 1,p \textrm{ prime}\\ip<11}}G_{i}^p(m).
\end{equation}

\begin{lemma}\label{basic22}Let $m$  be a positive integer with $m\geq 32$ and let $M$ be the symmetric group $\Sym(m)$ in its natural action on the $2$-subsets of $\{1,\ldots,m\}$. Then $|\mathcal{F}(M)|\leq G(m)$. (The function $G(m)$ is defined in Eq.~$\eqref{eq0000}$.)
\end{lemma}
\begin{proof}
We denote by $\Delta$ the set $\{1,\ldots,m\}$ and by $\Omega$ the set of $2$-subsets of $\Delta$. Fix $c$ with $c\leq m-11$.
For each prime $p$ and positive integer $i$ with $ip<11$, we set
$$\mathcal{C}_i^p=\{\langle g\rangle\in\mathcal{C}(M)\mid g\textrm{ has type }p^i \textrm{ on }\Delta\}.$$ 
Moreover, define
\begin{align*}
\mathcal{C}'&=\{\langle g\rangle\in\mathcal{C}(M)\mid \langle g\rangle\notin \mathcal{C}_i^p \textrm{ for every }p\textrm{ and }i,\,\mathrm{fix}_\Delta(g)\geq c\},\\
\mathcal{C}''&=\{\langle g\rangle\in\mathcal{C}(M)\mid\mathrm{fix}_\Delta(g)<c\}.
\end{align*}
By construction the sets $\mathcal{C}_i^p,\mathcal{C}',\mathcal{C}''$ are pair-wise disjoint, and $\mathcal{C}(M)=\bigcup_{p,i}\mathcal{C}_i^p\cup\mathcal{C}'\cup\mathcal{C}''$.

Clearly, $|\mathcal{C}_i^p|=m!/((p-1)(m-pi)!p^ii!)$. Moreover, if $C\in \mathcal{C}_i^p$, then $\fix_\Delta(C)=m-pi$,  and hence $\fix_\Omega(g)={m-pi\choose 2}$ when $p>2$ and $\fix_\Omega(g)={m-2i\choose 2}+i$ when $p=2$. Hence
\[
\mathrm{orb}_\Omega(C)=
\begin{cases}
{\frac{1}{p}\left({m\choose 2}-{m-pi\choose 2}\right)+{m-pi\choose 2}}&\textrm{when }p>2,\\
{\frac{1}{2}\left({m\choose 2}-{m-2i\choose 2}-i\right)+{m-2i\choose 2}+i}&\textrm{when }p=2.\\
\end{cases}
\]
Now, an easy computation gives $G_i^p(m)=\sum_{C\in \mathcal{C}_i^p}2^{\orb_\Omega(C)}$, see~\eqref{eq:def6}.

Write $\varepsilon=0.4$. Observe that $|\mathcal{C}'|\leq{m\choose c}(m-c)!=m!/c!<m^{m-c}=2^{(m-c)\log_2(m)}$ and that $|\mathcal{C}''|<m!<(\varepsilon \cdot m)^m$ for $m\geq 32$. 

Let $C$ be in $\mathcal{C}'$ and write $C=\langle g\rangle$. As $C\notin \mathcal{C}_i^p$ for every $p$ and $i$, we get $\fix_\Delta(g)\leq m-11$ and hence $\fix_\Omega(g)\leq {m-11\choose 2}$ if $|g|>2$. If $|g|=2$, then $\fix_\Delta(g)\leq m-12$ and hence $\fix_\Omega(g)\leq {m-12\choose 2}+6$. Now a computation shows that ${m-11\choose 2}\geq {m-12\choose 2}+6$ for $m\geq 18$. In particular, since we are assuming $m\geq 32$, in both cases $\fix_\Omega(C)\leq {m-11\choose 2}$. It follows that
$$\sum_{C\in\mathcal{C}'}2^{\orb_\Omega(C)}\leq 2^{\frac{1}{2}{m\choose 2}+\frac{1}{2}{m-11\choose 2}+(m-c)\log_2(m)}=2^{{m\choose 2}-\frac{11}{2}m+33+(m-c)\log_2(m)}.$$

Finally, let $C\in\mathcal{C}''$. Then $\fix_\Delta(C)\leq c-1$ and hence $\fix_\Omega(C)\leq{c-1\choose 2}+(m-c+1)/2$. It follows that
$$\sum_{C\in\mathcal{C}''}2^{\orb_\Omega(C)}\leq 2^{\frac{1}{2}{m\choose 2}+\frac{1}{2}\left({c-1\choose 2}+\frac{m-c+1}{2}\right)+m(\log_2(m)+\log_2(\varepsilon))}.$$

Now the lemma follows from the definition of $G(m)$ by taking $c=\lfloor m-2\log_2(m)\rfloor$ and by a careful computation.
\end{proof}


\section{Primitive groups of HC, CD and TW type}\label{HCCDPA}
In this section we prove Theorem~\ref{thrm:main} when $G$ is a primitive group on $\Omega$ of HC, CD or TW type. This case is already more complicated than the case discussed in Section~\ref{SD}, and presents all the main difficulties (but not the technicalities) of the remaining cases. We start by describing the structure and the action of the groups in these families: this will also  set the notation in the proof of Theorem~\ref{typeHS-SD-TW}.

 Assume that  $G$ is primitive of HC type (respectively, CD type) and let $N$ be the socle of $G$.  Then $G\leq H\wr L$ for some primitive group $H$  on $\Delta$ of HS type (respectively, SD type) and some transitive group $L$ of degree $\ell$. Moreover, the action of $G$ on $\Omega$ is equivalent to the product action of $G$ on $\Delta^\ell$.  Here $N$ equals the socle of $H\wr L$, the socle of $H$ is isomorphic to $T^a$ and $N$ is isomorphic to $T^{a\ell}$, for some non-abelian simple group $T$ and some positive integer $a\geq 2$ ($a=2$ when $H$ is of HS type). Finally, $|\Delta|=|T|^{a-1}$ and $|\Omega|=|T|^{(a-1)\ell}$.

Assume that  $G$ is primitive of TW type and let $N$ be the socle of $G$.  Then $G=N\rtimes L$ for some transitive group $L$ of degree $\ell$, and $N\cong T^\ell$ for some non-abelian simple group $T$  and some $\ell\geq 6$. The action of $G$ on $\Omega$ is equivalent to the natural ``affine'' action  of $G$ on $N$: the group $N$ acts on the set $N$ by right multiplication and $L$ acts by conjugation. Thus $|\Omega|=|T|^\ell$.

From~\cite{Pra}, we get that $\mathcal{M}(G)=\mathcal{M}_{\mathrm{PA}}(G)$ and the elements of $\mathcal{M}(G)$ permutation isomorphic to $\Sym(|\Omega|^{1/\ell'})\wr \Sym(\ell')$ (for some divisor $\ell'$ of $\ell$ with $\ell'>1$) are in one-to-one correspondence with the systems of imprimitivity of $L$ with $\ell'$ blocks of size $\ell/\ell'$. (This correspondence is natural: if $L$ has a system of imprimitivity with $\ell'$ blocks then the inclusion of $L$ in $\Sym(\ell/\ell')\wr\Sym(\ell')$ gives rise to  a natural inclusion of $G$ in $(H\wr \Sym(\ell/\ell'))\wr\Sym(\ell')\leq \Sym(|\Delta|^{\ell/\ell'})\wr\Sym(\ell')$.)  
Therefore, by Lemma~\ref{basic0}, we have 
\begin{equation}\label{eqTw}
|\mathcal{M}(G)|\leq \ell^{\log_2(\ell)}.
\end{equation}

\begin{theorem}\label{typeHS-SD-TW}
Let $G$ be a finite primitive group on $\Omega$ of HC, CD or TW type. Then there exists an edge-transitive hypergraph $\mathcal{H}=(\Omega,\mathcal{E})$ with $G=\Aut(\mathcal{H})$.
\end{theorem}
\begin{proof}
Write $n=|\Omega|$ and observe that $n\geq |T|^2\geq 60^2$. We use the notation that we established above. 
Clearly, $\ell\leq \log_{60}(n)$ and $n^{1/\ell'}\geq \sqrt{n}\geq |T|\geq 60>36$, for every divisor $\ell'$ of $\ell$ with $\ell'>1$. Therefore from Eq.~\eqref{eqTw} and Lemma~\ref{basic2}, we have
$$|\mathcal{S}(G)|\leq \sum_{M\in\mathcal{M}(G)}|\mathcal{F}(M)|\leq |\mathcal{M}(G)|F(n)\leq  \log_{60}(n)^{\log_2(\log_{60}(n))}F(n).$$
A computation shows that, for $n\geq 60^2$, the right hand side is strictly smaller than $2^n$ and hence $\mathcal{S}(G)\subsetneq 2^\Omega$. Now the proof follows from Lemma~\ref{basic-2}.
\end{proof}




\section{Primitive groups of AS type}\label{AStype}
In this section we prove Theorem~\ref{thrm:main} when $G$ is a finite primitive group on $\Omega$ of AS type. For this proof we use~\cite[Theorem~A]{Asch}. 

\begin{definition}\label{def1}{\rm
Following~\cite{Asch} (and also~\cite{LPS2,Pra}), we say that $G$ is \textit{product decomposable} if there exists a finite set $\Delta$, a positive integer $\ell$ with $\ell\geq 2$, and a subgroup $R$ of $H\wr \Sym(\ell)$ (endowed of its natural product action on $\Delta^\ell$) such that $G$ is permutation isomorphic to $R$. Now,~\cite[Theorem, part~II]{LPS1} says that, if $G$ is  product decomposable, then one of the following happens:
\begin{description}
\item[(i)]$\ell=2$, $\soc(G)=\Alt(6)$, $|\Delta|=6$, $|\Omega|=36$ and $G$ contains an outer-automorphism of $\Sym(6)$,
\item[(ii)]$\ell=2$, $G=\Aut(M_{12})$, $|\Delta|=12$ and $|\Omega|=144$,
\item[(iii)]$\ell=2$, $\soc(G)=\Sp_4(q)$, $q=2^k$ for some positive integer $k\geq 2$, $|\Delta|=q^2(q^2-1)/2$, $|\Omega|=q^4(q^2-1)^2/4$ and $G$ contains a graph automorphism of $\Sp_4(q)$.
\end{description} 
The group $G$ is \textit{product indecomposable} if it is not product decomposable. 
}
\end{definition}
\begin{theorem}\label{AS}
Let $G$ be a finite primitive group on $\Omega$ of AS type. Then either there exists an edge-transitive hypergraph $\mathcal{H}=(\Omega,\mathcal{E})$ with $G=\Aut(\mathcal{H})$, or $G$ is one of the groups in Table~$\ref{table1}$.
\end{theorem}
\begin{proof}
Let $T$ be the socle of $G$ and let $n$ be the degree of $G$. Then $T\unlhd G\leq \Aut(T)$.

If $T=\PSL_2(7)$ and $n=8$, then the proof follows with a computation. In particular, in what follows we may assume that $(T,n)\neq (\PSL_2(7),8)$. We start by dealing with the case that $G$ is product indecomposable. From~\cite[Theorem~A]{Asch}, we see that one of the following happens:
\begin{description}
\item[(1)]$|\mathcal{M}(G)|=1$;
\item[(2)]$G=T$, $|\mathcal{M} (T)| = 3$, $\Aut(T)\cong
\norm{\Sym(\Omega)}T\in \mathcal{M}(T)$, $\norm {\Sym(\Omega)}T$ is transitive on $\mathcal{M} (T)\setminus\{\norm {\Sym(\Omega)} T\}$, and $T$ is
maximal in $V$, where $K \in\mathcal{M} (T) \setminus \{\norm{\Sym(\Omega)} T\}$ and $V =\soc (K)$. Further $(T, V ,n)$ is one of the following:
\begin{description}
\item[(a)] $(HS, \Alt(m), 15400)$, where $m= 176$ and $n ={m\choose 2}$,
\item[(b)] $(\mathrm{G}_2(3),\mathrm{P}\Omega_7(3), 3159)$,
\item[(c)] $(\PSL_2(q),M_n,n)$, where $q \in \{11, 23\}$, $n =q +1$, and $M_n$ is the Mathieu group of degree $n$,
\item[(d)] $(\PSL_2(17), \Sp_8(2), 136)$;
\end{description}
\item[(3)] $T\cong \PSL_3(4)$ and $n=280$;
\item[(4)] $T\cong\Sz(q)$, $q = 2^k$, $k\geq 3$ is odd, $n = q^2(q^2 + 1)/2$, $\mathcal{M} (T) = \{K_1, K_2\}$ where $K_i = \norm {\Sym(\Omega)}{V_i}\cong \Aut(V_i)$, $V_1\cong \Alt(q^2+1)$ and the action of $V_1$ on $\Omega$ is equivalent to the action of $\Alt(q^2+1)$ on the $2$-subsets of $\{1,\ldots,q^2+1\}$, $V_2=\Sp_{4k}(2)$, and $\norm {\Sym(\Omega)} T\cong \Aut(T)$ is maximal in $V_1$;
\item[(5)] $G=\PSL_2(11)$ and $n = 55$.
\end{description} 
We deal with each of these possibilities in a case-by-case basis. If Case~(1) holds, then the proof follows from Lemma~\ref{basic}. If Case~(2c),~(2d),~(3) or~(5) holds, then the proof follows with a computation with \texttt{magma}. 

Suppose that Case~(2b) holds. Then $\mathcal{M}(G)=\{K_0,K_1,K_2\}$, with $K_0\cong \Aut(\mathrm{G}_2(3))$ and $K_1\cong K_2\cong \Aut(\mathrm{P}\Omega_7(3))$.  From~\cite[Corollary~$1$]{GM}, we get that $\fix_\Omega(g)\leq 4n/7$ for every non-identity element $g$ in $K_i$, for $i\in \{1,2,3\}$. It follows from Lemma~\ref{fact3} that
$$|\mathcal{S}(G)|\leq |\mathcal{F}(K_0)|+|\mathcal{F}(K_1)|+|\mathcal{F}(K_2)|\leq(|K_0|+|K_1|+|K_2|)\cdot 2^{\frac{n}{2}+\frac{2n}{7}}.$$ A computation shows that $|\mathcal{S}(G)|<2^n$ and hence $\mathcal{S}(G)\subsetneq 2^\Omega$. Thus the proof follows from Lemma~\ref{basic-2}.

Suppose that Case~(2a) holds. Then $\mathcal{M}(G)=\{K_0,K_1,K_2\}$, with $K_0=\norm{\Sym(\Omega)}G\cong \Aut(HS)$ and $K_1\cong K_2\cong \Sym(176)$.  From~\cite[Corollary~$1$]{GM}, we get that $\fix_\Omega(g)\leq 4n/7$ for every  $g\in K_0\setminus\{1\}$. Therefore, from Lemmas~\ref{fact3} and~\ref{basic22}, we get $|\mathcal{S}(G)|\leq 2^{\frac{n}{2}+\frac{2n}{7}}\cdot |K_0|+2\cdot G(176)$. Now a computation shows that $|\mathcal{S}(G)|<2^n$, hence $\mathcal{S}(G)\subsetneq 2^\Omega$ and we conclude using Lemma~\ref{basic-2}.

Suppose that Case~(4) holds. Then $\mathcal{M}(G)=\{K_1,K_2\}$, with $K_1\cong \Sym(q^2+1)$, $K_2\cong \Sp_{4k}(2)$ and the action of $K_1$ on $\Omega$ is equivalent to the action of $\Sym(q^2+1)$ on the $2$-subsets of $\{1,\ldots,q^2+1\}$.   From~\cite[Corollary~$1$]{GM}, we get that $\fix_\Omega(g)\leq 4n/7$ for every  $g\in K_2\setminus\{1\}$. Therefore, from Lemmas~\ref{fact3} and~\ref{basic22}, we get $|\mathcal{S}(G)|\leq  G(q^2+1)+2^{\frac{n}{2}+\frac{2n}{7}}|K_2|.$ Using $q^2+1\geq 65$, a computation shows that $|\mathcal{S}(G)|<2^n$, hence $\mathcal{S}(G)\subsetneq 2^\Omega$ and we conclude using Lemma~\ref{basic-2}.

\smallskip

Suppose that $G$ is product decomposable. If $T=\Alt(6)$ or $T=M_{12}$, then the proof follows with a computation in \texttt{magma}. Therefore we are left with $T=\Sp_4(q)$, $q=2^k$, $k\geq 2$, $n=(q^2(q^2-1)/2)^2$ and $G$ is contained in a wreath product $\Sym(m)\wr\Sym(2)$ with $m=q^2(q^2-1)/2$. Here we refer to~\cite[Section~$4$]{LPS1} for the information on this permutation representation. 
From~\cite[Theorem and Tables~III,~IV,~V]{LPS1}, we deduce that $\mathcal{M}_{\mathrm{HA}}(G)=\emptyset$ and $\mathcal{M}_{\mathrm{AS}}(G)=\{K\}$ where $K=\norm{\Sym(\Omega)} T\cong \Aut(T)$.
 We claim that there exists a unique $G$-invariant Cartesian decomposition of $\Omega$ and hence $|\mathcal{M}_{\mathrm{PA}}(G)|=1$ by Lemma~\ref{CDlemma}. Let $\Omega=\Delta_1\times \Delta_2$ be a $G$-invariant Cartesian decomposition. The group $T=\soc(G)$ is imprimitive and has two systems of imprimitivity with $m$ blocks of size $m$ (namely, $\{\Delta_1\times\{\delta\}\mid \delta\in \Delta_2\}$ 
 and $\{\{\delta\}\times \Delta_2\mid\delta\in\Delta_1\}$), which are interchanged by  $G$. For $\omega=(\delta_1,\delta_2)\in \Omega$, from~\cite[Section~$4$]{LPS1}, we see that $T_\omega\cong C_{q^2+1}\rtimes C_4$ is the normaliser of a torus of order $q^2+1$. Moreover, $T_{\Delta_1\times \{\delta_2\}}\cong T_{\{\delta_1\}\times\Delta_2}\cong \mathrm{O}_4^-(q)\cong \mathrm{SL}_2(q^2).2\cong \Sp_2(4).2$ and $T_{\Delta_1\times\{\delta_2\}}$, $T_{\{\delta_1\}\times\Delta_2}$ are maximal subgroups of $T$.  From~\cite[Section~$4.8$]{KL} or the discussion in~\cite[Section~4]{LPS1}, we see that $T$ has exactly two conjugacy classes of subgroups isomorphic to $\mathrm{O}_4^-(q)$: one conjugacy class with representative the stabiliser of a quadratic form for the underlying vector space of $T$ and one with representative the stabiliser of an extension field. These two classes are fixed by the subgroup of index $2$ of $\Aut(T)$ consisting of the inner-diagonal and field automorphisms, and are fused by the remaining elements. Furthermore, from the list of maximal subgroups of $\Sp_4(q)$ in~\cite{kleidman} (see also~\cite{BCH}), we see that if $U$ is any subgroup of $T$ with $|T:U|=m$, then $U$ is maximal in $T$ and is conjugate to either $T_{\Delta_1\times\{\delta_2\}}$ or  $T_{\{\delta_1\}\times\Delta_2}$. Therefore $\{\Delta_1\times\{\delta\}\mid \delta\in \Delta_2\}$ 
 and $\{\{\delta\}\times \Delta_2\mid\delta\in\Delta_1\}$ are the only systems of imprimitivity of $T$ with $m$ blocks of size $m$ and our claim follows.  
 Therefore $\mathcal{M}(G)=\{K,M\}$, where $M\cong\Sym(m)\wr\Sym(2)$.  From~\cite[Corollary~$1$]{GM}, we have $\fix_\Omega(g)\leq 4n/7$ for every $g\in K$ with $g\neq 1$. Thus
$$|\mathcal{S}(G)|\leq |\mathcal{F}(K)|+|\mathcal{F}(M)|\leq  2^{\frac{n}{2}+\frac{2n}{7}}|K|+F(n).$$
A computation shows that $|\mathcal{S}(G)|<2^n$, hence $\mathcal{S}(G)\subsetneq 2^\Omega$ and we conclude using Lemma~\ref{basic-2}.
\end{proof}


\section{Primitive groups of HA type}\label{HAtype}

In this section we prove Theorem~\ref{thrm:main} when $G$ is a primitive group on $\Omega$ of HA type.  We start with a number-theoretic remark.
\begin{lemma}\label{nt}Let $p$ be a prime number. Then there are at most $(p-1)/2$ solutions to the equation $p=(q^\ell-1)/(q-1)$ where $\ell$ is a positive integer with $\ell\geq 2$ and $q$ is a prime power.
\end{lemma}
\begin{proof}
It is clear that if $(q^\ell-1)/(q-1)=p=(q'^\ell-1)/(q'-1)$, then $q=q'$. Observe that if $(q^\ell-1)/(q-1)$ is prime, then $\ell$ is prime and $\ell$ divides $p-1$. Now the proof follows immediately.
\end{proof}
Lemma~\ref{nt} is far from best possible and should not be taken too seriously. Nevertheless, the equation $p=(q^\ell-1)(q-1)$ can have more than one solution. For instance, $(5^3-1)/(5-1)=31=(2^5-1)/(2-1)$.

\begin{theorem}\label{typeHA}
Let $G$ be a finite primitive group on $\Omega$ of HA type. Then either there exists an edge-transitive hypergraph $\mathcal{H}=(\Omega,\mathcal{E})$ with $G=\Aut(\mathcal{H})$, or $G$ is one of the groups in Table~$\ref{table1}$.
\end{theorem}
\begin{proof}
Let $V$ be the socle of $G$ and let $H$ be the stabiliser of a point of $\Omega$. Then $V$ is an elementary abelian $p$ group of size $p^d$, for some prime number $p$ and some positive integer $d$. Moreover, $G=V\rtimes H$ and the action of $G$ on $\Omega$ is equivalent to the ``affine'' action of $G$ on $V$, that is, we may identify $G$ with a subgroup of $\AGL_d(p)$ and $H$ with an irreducible subgroup of $\GL_d(p)$. Write $K=\norm{\Sym(\Omega)}V$ and observe that $K\cong\AGL_d(p)$.

From~\cite{Pra}, we see that every element of $\mathcal{M}(G)$ is either of HA, PA or AS type, and that $\mathcal{M}_{\mathrm{HA}}(G)=\{K\}$. If $\mathcal{M}(G)=\mathcal{M}_{\mathrm{HA}}(G)$, then the proof follows from Lemma~\ref{basic}. Suppose then that $\mathcal{M}_{\mathrm{AS}}(G)\cup\mathcal{M}_{\mathrm{PA}}(G)\neq \emptyset$.
Observe that, for every $g\in K$ with $g\neq 1$, we have $\fix_\Omega(g)\leq n/p$ and hence, by Lemma~\ref{basic3}, $\orb_\Omega(g)\leq n/2+n/{2p}=(p+1)n/(2p)$. Therefore 
\begin{equation}\label{eq:def7}
|\mathcal{F}(K)|\leq 2^{\frac{(p+1)n}{2p}}|K|.
\end{equation}

Suppose that $\mathcal{M}_{\mathrm{AS}}(G)\neq \emptyset$ and let  $M\in\mathcal{M}_{\mathrm{AS}}(G)$. From~\cite[Proposition~$5.1$]{Pra}, we get that either 
\begin{description}
\item[(i)]$|V|\in \{11,23,27\}$, or
\item[(ii)] $|V|=p=(q^{\ell}-1)/(q-1)$ for some prime power $q$ and some $\ell\geq 2$, $M\cong \PGammaL_\ell(q)$ and the action of $M$ on $V$ in equivalent to the  natural $2$-transitive action of $\PGammaL_\ell(q)$ on the points (or hyperplanes) of the $d$-dimensional  projective space over the finite field of size $q$.
\end{description}
 When Case~(i) holds,  a computation with \texttt{magma} shows that $G$ is the automorphism group of an edge-transitive hypergraph $\mathcal{H}=(\Omega,\mathcal{E})$. Then suppose that Case~(ii)  holds. Since $|\Omega|=|V|=p$ is square-free, $\Sym(\Omega)$ has no primitive subgroups of PA type, that is, $\mathcal{M}_{\mathrm{PA}}(G)=\emptyset$ and $\mathcal{M}(G)=\{K\}\cup \mathcal{M}_{\mathrm{AS}}(G)$. Let $M_1,\ldots,M_s$ be  representatives for the $\Sym(\Omega)$-conjugacy classes of the elements of $\mathcal{M}_\mathrm{AS}(G)$. Thus $M_i\cong \PGammaL_{\ell_i}(q_i)$, for some $\ell_i\geq 2$ and some prime power $q_i=r_i^{f_i}$, and $(q_i,\ell_i)\neq (q_j,\ell_j)$ when $i\neq j$.  Moreover, $$\mathcal{M}_{\mathrm{AS}}(G)=\bigcup_{i=1}^s\{M_i^g\mid g\in \Sym(\Omega), G\leq M_i^g\}.$$ For each $i\in \{1,\ldots,s\}$, write $t_i=|\{M_i^g\mid g\in\Sym(\Omega),G\leq M_i^g\}|$ and let $g_{i,1},\ldots,g_{i,t_i}\in \Sym(G)$ such that $$\{M_i^g\mid g\in\Sym(\Omega),G\leq M_i^g\}=\{M_i^{g_{i,1}},\ldots,M_i^{g_{i,t_i}}\}.$$ 
 We claim that $t_i\leq (p-1)/(\ell_i f_i)$. Observe that $V^{g_{i,j}^{-1}}$ is a Sylow $p$-subgroup of $M_i$, for every $j$. Therefore, for every $j$, there exists $a_{j}\in M_i$ with $V^{g_{i,j}^{-1}a_j}=V$ and  hence $g_{i,j}\in M_i\norm {\Sym(\Omega)}V=M_iK$. Therefore $t_i$ is at most the number of $M_i$-cosets contained in $M_iK$, that is, $t_i\leq|M_iK:M_i|=|K:\norm {M_i}V|=p(p-1)/(p\ell_i f_i)$ and our claim is proved.
 
Write 
\[m_i=
\begin{cases}
\frac{4p}{3q_i}&\textrm{if }\ell_i>2,\\
\max\left\{\frac{4p}{3q_i},2,q_i^{1/e}+1\mid e \textrm{ divides } f_i,\,e>1\right\}&\textrm{if } \ell_i=2.
\end{cases}
\]
From~\cite[Theorem~$1$]{LieSax}, we see that for every $g\in M_i$ with $g\neq 1$  we have $\fix_\Omega(g)\leq m_i$ and hence $\orb_\Omega(g)\leq (p+m_i)/2$. It follows that
\begin{eqnarray*}
\sum_{M\in \mathcal{M}_{\mathrm{AS}}(G)}|\mathcal{F}(M)|&\leq&\sum_{i=1}^s2^{(p+m_i)/2}|M_i|(p-1)/(\ell_i f_i).
\end{eqnarray*}
Now, $|M_i|<p^{\log_2(p)}$, $\ell_i\geq 2$,  $m_i\leq \sqrt{p}+1$ and $s\leq (p-1)/2$ by  Lemma~\ref{nt}. Using these inequalities and Eq.~\eqref{eq:def7}, we get
\begin{eqnarray*}
|\mathcal{S}(G)|&\leq& 2^{(p+1)/2}p(p-1)+2^{(p+\sqrt{p}+1)/2}p^{\log_2(p)}(p-1)^2/4.
\end{eqnarray*}
For $p\geq 139$, a computation (with the help of a computer) shows that the right hand side of this equation is strictly less than $2^{p}$. In particular, when $p\geq 139$, the proof follows from Lemma~\ref{basic-2}.  The only primes less than $139$ of the form $(q^\ell-1)/(q-1)$ with $q$ a prime power and $\ell\geq 2$ are:
$$5,7,13,17,31,73,127.$$  
For these values of $p$ we can explicitly construct with \texttt{magma} an edge-transitive hypergraph  $\mathcal{H}=(\Omega,\mathcal{E})$ such that $G=\Aut(\mathcal{H})$, except when $G$ is one of the groups in Table~\ref{table1}.

\smallskip

Suppose now that $\mathcal{M}_{\mathrm{AS}}(G)=\emptyset$. Thus $\mathcal{M}(G)=\{K\}\cup\mathcal{M}_{\mathrm{PA}}(G)$ and $\mathcal{M}_{\mathrm{PA}}(G)\neq \emptyset$. In particular, $H$ is an imprimitive linear group. Given $\ell\in \{0,\ldots,d\}$, we write 
\[
\left[
\begin{array}{c}
d\\
\ell
\end{array}
\right]_p=\frac{(p^d-1)(p^{d-1}-1)\cdots (p^{d-\ell+1}-1)}{(p^\ell-1)(p^{\ell-1}-1)\cdots (p-1)}.\]
From Lemma~\ref{CDlemma}, the elements of $\mathcal{M}_{\mathrm{PA}}(G)$ are in one-to-one correpondence with the $H$-invariant Cartesian
 decompositions $V=V_1\times \cdots \times V_\ell$ of $V$, where $\ell\geq 2$ and $|V_1|=\cdots =|V_\ell|\geq 5$. Let us denote by $c_\ell$ the number of $H$-invariant Cartesian
  decompositions $V=V_1\times \cdots \times V_\ell$. 
Observe that, for every $i\in \{1,\ldots,\ell\}$, the  $H$-invariant Cartesian decomposition $V=V_1\times \cdots \times V_\ell$ is uniquely determined by  $V_i$: in fact $V_1,\ldots,V_\ell$ can be reconstructed from the $H$-orbit of $V_i$, that is, $\{V_i^h\mid h\in H\}=\{V_1,\ldots,V_\ell\}$. This shows that $c_\ell$  is at most the number of $H$-orbits  of cardinality $\ell$ in the action on the subspaces of $V$ of dimension $d/\ell$, that is, $$c_\ell\leq \frac{1}{\ell}\left[
\begin{array}{c}
d\\
d/\ell
\end{array}
\right]_p.$$
Thus 
$$|\mathcal{M}_{\mathrm{PA}}(G)|\leq\sum_{\substack{\ell\mid d,\,\ell>1\\p^{d/\ell}\geq 5}}\frac{1}{\ell}\left[
\begin{array}{c}
d\\
d/\ell
\end{array}
\right]_p$$
and
\begin{eqnarray}\label{HA1}
|\mathcal{S}(G)|&\leq& 2^{\frac{p+1}{2p}n}\cdot |K|+\sum_{\substack{\ell\mid d\,\ell>1\\p^{d/\ell}\geq 5}}\frac{1}{\ell}\left[
\begin{array}{c}
d\\
d/\ell
\end{array}
\right]_p\mathcal{F}\left(\Sym(p^{d/\ell})\wr\Sym(\ell)\right).
\end{eqnarray}
As $|K|=|\AGL_d(p)|<n^{1+\log_2(n)}$ and $V$ has at most $|V|^{d/2}$ subspaces of dimension at most $d/2$, using Lemma~\ref{basic2} we get
$$|\mathcal{S}(G)|<2^{3n/4}\cdot n^{1+\log_2(n)}+n^{\log_2(n)/2}\cdot F(n).$$
A computation gives that the right hand side of this inequality is less than $2^n$ for $n\geq 10533$. Thus for $n\geq 10533$, we have $\mathcal{S}(G)\subsetneq 2^\Omega$ and we conclude using Lemma~\ref{basic-2}. Similarly, for every $p$ and $d$ with $n=p^d<10533$, we compute the exact value of the right hand side of Eq.~\eqref{HA1} and check when it is strictly less than $2^n$. For these values of $p$ and $d$ the proof follows again from Lemma~\ref{basic-2}. The remaining values are: $d=2$ and $5\leq p\leq 31$, $d=3$ and $p\in \{5,7\}$, $d=4$ and $p\in\{3, 5\}$, $d=6$ and $p\in \{2,3\}$, $d\in\{8,9, 10\}$ and $p=2$. 

For each of the remaining values of $p$ and $d$, and for each divisor $\ell$ of $d$ with $\ell>1$ and $p^{d/\ell}\geq 5$, we may compute explicitly with a computer the value of $$\sum_{C\in \mathcal{C}(\Sym(p^{d/\ell})\wr\Sym(\ell))}2^{\orb_\Omega(C)}$$ and use it (in view of Lemma~\ref{fact3}) as an upper bound for $\mathcal{F}(\Sym(p^{d/\ell})\wr\Sym(\ell))$ in Eq.~\eqref{HA1}. With this improvement we get $|\mathcal{S}(G)|<2^n$, except when $d=2$ and $p\in \{5,7\}$, $d=4$ and $p=3$, $d\in \{6,8\}$ and $p=2$. Finally each affine primitive group with one of these degrees can be checked directly with \texttt{magma}.
\end{proof}

 \section{Primitive groups of PA type}\label{PAtype}

Finally, in this section we prove Theorem~\ref{thrm:main} when $G$ is a primitive group on $\Omega$ of PA type. We start with a preliminary lemma.
\begin{lemma}\label{CartProduc}Let $G$ be a finite primitive group of degree $n$ on $\Omega$. Then $|\mathcal{M}_{\mathrm{PA}}(G)|\leq n^{\log_2(n)/2}$. 
\end{lemma}
\begin{proof}
From Lemma~\ref{CDlemma}, it suffices to show that $\Omega$ admits at most $n^{\log_2(n)/2}$ $G$-invariant Cartesian decompositions. 

Let $N$ be the socle of $G$. Let $\Omega=\Omega_1\times \cdots \times \Omega_\ell$ be a $G$-invariant Cartesian decomposition with $\ell\geq 2$. For each $i\in\{1,\ldots,\ell\}$, set 
$$\mathcal{C}_i=\{\{(\omega_1,\ldots,\omega_{i-1},\varepsilon,\omega_{i+1},\ldots,\omega_{\ell})\mid \varepsilon\in \Omega_i\}\mid \omega_j\in \Omega_j\textrm{ for each }j\in \{1,\ldots,\ell\}\setminus\{i\}\}.$$
Observe that $\mathcal{C}_1,\ldots,\mathcal{C}_\ell$ are systems of imprimitivity for $N$ (with $|\Omega|^{(\ell-1)/\ell}$ blocks of size $|\Omega|^{1/\ell}$) and that $G$ acts transitively on $\{\mathcal{C}_1,\ldots,\mathcal{C}_\ell\}$. In particular, $\mathcal{C}_1,\ldots,\mathcal{C}_\ell$ (and hence the decomposition $\Omega_1\times \cdots\times \Omega_\ell$) is uniquely determined by $\mathcal{C}_1$ because $\mathcal{C}_1^G=\{\mathcal{C}_1^g\mid g\in G\}=\{\mathcal{C}_1,\ldots,\mathcal{C}_\ell\}$. 

This shows that $|\mathcal{M}_{\mathrm{PA}}(G)|$ is at most the number of systems of imprimitivity for $N$ with blocks of size at most $\sqrt{n}$. Now the proof follows by the proof of Lemma~\ref{basic0}.
\end{proof}

Again, Lemma~\ref{CartProduc} should not be taken too seriously, but it is perfect for our application in the proof of Theorem~\ref{typePA}. 
\begin{theorem}\label{typePA}
Let $G$ be a finite primitive group on $\Omega$ of PA type. Then there exists an edge-transitive hypergraph $\mathcal{H}=(\Omega,\mathcal{E})$ with $G=\Aut(\mathcal{H})$.
\end{theorem}
 \begin{proof}
 Let $N$ be the socle of $G$ and let $n$ be the cardinality of $\Omega$. Then $N=T_1\times \cdots \times T_\ell$, where $T_1,\ldots,T_\ell$ are pair-wise isomorphic non-abelian simple groups and $\ell\geq 2$. Moreover, $G\leq H\wr L$, $\Omega=\Delta^\ell$, $H$ is a primitive group on $\Delta$ of AS type with socle isomorphic to $T_1$, $L$ is transitive of degree $\ell$ and the action of $G$ on $\Omega$ is the natural product action of $G$ on $\Delta^\ell$.

Suppose that $(H,|\Delta|)\neq (\PSL_2(7),8)$ and that $H$ is product indecomposable. Then from~\cite[Section~$7$]{Pra},  $\mathcal{M}(G)=\mathcal{M}_{\mathrm{PA}}(G)$  
and the elements of $\mathcal{M}_{\mathrm{PA}}(G)$  permutation isomorphic to the wreath product $\Sym(|\Omega|^{1/\ell'})\wr \Sym(\ell')$ (for some divisor $\ell'$ of $\ell$ with $\ell'>1$) are in one-to-one correspondence with the systems of imprimitivity of $L$ with $\ell'$ blocks of size $\ell/\ell'$. (Exactly as in Section~\ref{HCCDPA} for groups of HC, CD and TW type, this correspondence is natural: if $L$ has a system of imprimitivity with $\ell'$ blocks then the inclusion of $L$ in $ \Sym(\ell/\ell')\wr\Sym(\ell')$ gives rise to  a natural inclusion of $G$ in $(H\wr \Sym(\ell/\ell'))\wr\Sym(\ell')\leq \Sym(|\Delta|^{\ell/\ell'})\wr\Sym(\ell')$.)  
By Lemmas~\ref{basic0} and~\ref{basic2}, we have 
$|\mathcal{M}(G)|\leq \ell^{\log_2(\ell)}$ and
\begin{equation}\label{PAeq}
|\mathcal{S}(G)|\leq \ell^{\log_2(\ell)}F(n).
\end{equation}

Suppose that $\ell$ is prime. Then $G$ is contained in a unique maximal subgroup of $\Sym(\Omega)$ (namely $\Sym(\Delta)\wr\Sym(\ell)$) and the proof follows from Lemma~\ref{basic}. In particular, we may assume that $\ell\geq 4$. Observe that $\ell\leq\log_5(n)$. Using this upper bound for $\ell$ and Eq.~\eqref{PAeq}, it follows that $|\mathcal{S}(G)|<2^{n}$ for $n\geq 1290$. Therefore when $n\geq 1290$ we conclude by Lemma~\ref{basic-2}. Assume that $n\leq 1289$. Observe that  $5\leq |\Delta|=n^{1/\ell}\leq n^{1/4}<6$ and $1290^{1/5}<5$. Hence $|\Delta|=5$, $\ell=4$ and $N\cong \Alt(5)^4$. Now, the primitive groups of PA type with socle  $\Alt(5)^4$ and degree $5^4$ can be checked directly with \texttt{magma}.

\smallskip
   
Suppose that $H=\PSL_2(7)$ and $|\Delta|=8$. Now,  from~\cite[Section~$7$]{Pra},  $\mathcal{M}(G)=\mathcal{M}_{\mathrm{HA}}(G)\cup \mathcal{M}_{\mathrm{PA}}(G)$. Moreover, since $H$ is product indecomposable, we have $|\mathcal{M}_{\mathrm{PA}}(G)|\leq \ell^{\log_2(\ell)}$ as in the previous case. 
In particular, when $\ell\geq 3$, Lemma~\ref{basic2} gives
\begin{equation}\label{eq:def11}
\sum_{M\in\mathcal{M}_{\mathrm{PA}}(G)}|\mathcal{F}(M)|\leq \ell^{\log_2(\ell)}F(8^\ell).
\end{equation}
  
Let $K\in \mathcal{M}_{\mathrm{HA}}(G)$. Then $K\cong \AGL_{3\ell}(2)$ by ~\cite[Proposition~$7.1$]{Pra}. As $\Sym(\Omega)$ contains a unique conjugacy class of primitive groups isomorphic to $\AGL_{3\ell}(2)$, we have $\mathcal{M}_{\mathrm{HA}}(G)=\{K^g\mid g\in \Sym(\Omega), G\leq K^g\}$. Write $t=|\mathcal{M}_{\mathrm{HA}}(G)|$. We prove the following inequality: 
\begin{equation}\label{eq:def10}
t\leq 16^\ell\frac{|\GL_{3\ell}(2)|}{|\GL_3(2)\wr\Sym(\ell)|}|\mathrm{PGL}_2(7)\wr\Sym(\ell)|.
\end{equation}
(Observe that the right hand side of Eq.~\eqref{eq:def10} is simply $32^\ell\cdot|\GL_{3\ell}(2)|$.)
  We argue by contradiction and we assume that Eq.~\eqref{eq:def10} is false. Let $g_1,\ldots,g_t\in \Sym(\Omega)$ such that $\mathcal{M}_{\mathrm{HA}}(G)=\{K^{g_1},\ldots,K^{g_t}\}$. In particular, $G^{g_1^{-1}},\ldots,G^{g_t^{-1}}\leq K$ and hence $N^{g_1^{-1}},\ldots, N^{g_t^{-1}}\leq K$. Let $V$ be the socle of $K$, thus $V$ is an elementary abelian group of order $2^{3\ell}$. As $N^{g_i^{-1}}\cap V\unlhd G^{g_i^{-1}}$, we have $N^{g_i^{-1}}\cap V=1$ and hence $VN^{g_i^{-1}}\cong V\rtimes N^{g_i^{-1}}$. As $G^{g_i^{-1}}\leq K$, we may identify the $G^{g_i^{-1}}$-set $\Omega$ with $V$. Thus, since $N^{g_i^{-1}}$ fixes each direct factor of a Cartesian decomposition  of $\Omega$, we may write $V=V_1\times \cdots \times V_\ell$ where $V_1,\ldots,V_\ell$ are $N^{g_i^{-1}}$-invariant subspaces of $V$ of dimension $3$. Observe that $V$ has $|\GL_{3\ell}(2):\GL_3(2)\wr\Sym(\ell)|$ Cartesian decompositions $V=V_1\times \cdots\times V_\ell$ with $\dim V_{1}=\cdots =\dim V_\ell=3$. In particular, as Eq.~\eqref{eq:def10}  is false, from the pigeon-hole principle there exists a Cartesian decomposition $V=V_1\times \cdots \times V_\ell$ (with $\dim V_i=3$ for each $i$) and a subset $X$ of $\{1,\ldots,t\}$ such that $|X|>16^\ell|\PGL_2(7)\wr\Sym(\ell)|$ and $N^{g_{x}^{-1}}$ fixes each of the subspaces $V_1,\ldots,V_\ell$ for every $x\in X$. Therefore, for every $x,y\in X$, we have $VN^{g_x^{-1}}=VN^{g_y^{-1}}$.

Now, $VN^{g_i^{-1}}=(V_1 T_1^{g_i^{-1}})\times \cdots \times (V_\ell T_\ell^{g_i^{-1}})\cong \AGL_3(2)^\ell$. A computation shows that $\AGL_3(2)$ contains $16$ transitive subgroups isomorphic to $\PSL_2(7)$. Therefore $VN^{g_i^{-1}}$ contains $16^\ell$ transitive subgroups isomorphic to $\PSL_2(7)^\ell$. Thus, from the pigeon-hole principle, there exists a subset $Y$ of $X$ such that $|Y|>|\PGL_2(7)\wr\Sym(\ell)|$ and $N^{g_y^{-1}}=N^{g_z^{-1}}$, for every $y,z\in Y$. Fix $y_0\in Y$. Now, $g_{y_0}^{-1}{g_y}\in \norm{\Sym(\Omega)}N$ for every $y\in Y$. As $|\norm{\Sym(\Omega)}N|=|\PGL_2(7)\wr\Sym(\ell)|$, from the pigeon-hole principle we  have  $g_{y_0}^{-1}g_y=g_{y_0}^{-1}g_z$ for some $y,z\in Y$ with $y\neq z$. Thus $g_y=g_z$, a contradiction. Now Eq.~\eqref{eq:def10} is proved.

Observe that every element of $K\cong \AGL_{3\ell}(2)$ fixes at most $n/2$ points and hence  by Lemma~\ref{fact3} $|\mathcal{F}(K)|\leq 2^{3n/4}|K|$.

Now, from Eqs.~\eqref{eq:def11} and~\eqref{eq:def10}, for $\ell\geq 3$ we get
\begin{eqnarray*}
|\mathcal{S}(G)|&\leq& \sum_{M\in \mathcal{M}(G)}|\mathcal{F}(M)|\leq |\mathcal{F}(K)||\mathcal{M}_{\mathrm{HA}}(G)|+F(n)|\mathcal{M}_{\mathrm{PA}}(G)|\leq2^{3n/4}|K|32^\ell+\ell^{\log_2(\ell)}F(8^{\ell}).
\end{eqnarray*}   
A computation shows that the right hand side of this equation is less that $2^n$ when $\ell\geq 4$. In particular, when $\ell\geq 4$, we conclude with Lemma~\ref{basic-2}. Finally, the primitive groups of PA type with socle $\PSL_2(7)^2$ (respectively $\PSL_2(7)^3$) and degree $8^2$ (respectively $8^3$) can be (as usual) checked with \texttt{magma}.
     
\smallskip

Finally suppose that $H$ is product decomposable. Definition~\ref{def1} gives  $T_1\cong\Alt(6)$ and $|\Delta|=36$, or $T_1\cong M_{12}$ and $|\Delta|=144$, or $T_1\cong \Sp_4(q)$ with $q>2$ even and $|\Delta|=(q^2(q^2-1)/2)^2\geq (4^2\cdot(4^2-1)/2)^2=14400$. From~\cite[Proposition~$7.1$]{Pra} we get $\mathcal{M}(G)=\mathcal{M}_{\mathrm{PA}}(G)$ and  from Lemma~\ref{CartProduc} we get $|\mathcal{M}_{\mathrm{PA}}(G)|\leq
 n^{\log_2(n)/2}$. Thus 
 $|\mathcal{S}(G)|\leq n^{\log_2(n)/2}F(n)$. A computation shows that, for $n\geq 10533$,  $n^{\log_2(n)/2}F(n)<2^n$ and hence for these values of $n$ the proof follows from Lemma~\ref{basic-2}. Assume that $n<10533$. As $10533^{1/3}<36$, we must have $\ell=2$. Then $|\Delta|<10533^{1/2}$ and hence $|\Delta|\leq 102$. Thus $(T_1,|\Delta|,\ell,n)= (\Alt(6),36,2,36^2)$. A computation with \texttt{magma} shows that $\Omega$ has $4$ $N$-invariant Cartesian decompositions. As $4\cdot F(n)<2^{n}$, we conclude again with Lemma~\ref{basic-2}.
\end{proof}

\section{Concluding remarks}\label{theend}

We finish by bringing together the various threads to prove Theorem~\ref{thrm:main}.
\begin{proof}[Proof of Theorem~$\ref{thrm:main}$]
The proof follows immediately from Theorems~\ref{typeHS-SD},~\ref{typeHS-SD-TW},~\ref{AS},~\ref{typeHA} and~\ref{typePA}.
\end{proof}


\thebibliography{10}

\bibitem{Asch1}M.~Aschbacher, Overgroups of primitive groups, \textit{J. Aust. Math. Soc.} \textbf{87} (2009), 37--82.

\bibitem{Asch}M.~Aschbacher, Overgroups of primitive groups II, \textit{J. Algebra} \textbf{322} (2009), 1586--1626.

\bibitem{Ba1}L.~Babai, The probability of generating the symmetric group, \textit{J. Comb. Theory Ser. A}\textbf{52} (1989), 148--153.

\bibitem{BC}L.~Babai, P.~J.~Cameron, Most primitive groups are full automorphism groups of edge-transitive hypergraphs, \href{http://arxiv.org/pdf/1404.6739v1.pdf}{arXiv:1404.6739}.


\bibitem{magma} W.~Bosma, J.~Cannon, C.~Playoust, The Magma algebra system. I. The user language, \textit{J. Symbolic Comput.} \textbf{24} (3-4) (1997), 235--265. 

\bibitem{BCH}J.~Bray, D.~Holt, C.~Roney-Dougal, \textit{The maximal subgroups of the low-dimensional classical groups}, London Mathematical Society Lecture Notes Series \textbf{407}, Cambridge University Press, Cambridge, 2013.



\bibitem{CNS}P.~J.~Cameron, P.~M.~Neumann, J.~Saxl, On groups with no regular orbits on the set of subsets, \textit{Arch. Math.} \textbf{43} (1984), 295--296.

\bibitem{ATALS}J.~H.~Conway, R.~T.~Curtis, S.~P.~Norton, R.~A.~Parker, R.~A.~Wilson, \textit{The Atlas of finite groups}, Clarendon Press, Oxford, 1985. 

\bibitem{DVS}F.~Dalla Volta, J.~Siemons, Orbit equivalence and permutation grousp defined by unordered relations, \textit{J. Algebr. Comb.} \textbf{35} (2012), 547--564.

\bibitem{DM}J.~D.~Dixon, B.~Mortimer, \textit{Permutation Groups}, Graduate Texts in Mathematics \textbf{163}, Springer-Verlag, New York, 1996.

\bibitem{GM}R.~Guralnick, K.~Magaard, On the minimal degree of a primitive permutation group, \textit{J. Algebra} \textbf{207} (1998), 127--145.

\bibitem{Kantor}W.~M.~Kantor, $k$-homogeneous groups, \textit{Math. Z.} \textbf{124}, 261--265.

\bibitem{kleidman}P.~B.~Kleidman, \textit{The subgroup structure of some finite simple groups}, Ph.D. thesis, University of Cambridge,
1987.

\bibitem{KL}P.~Kleidman, M.~Liebeck, \textit{The subgroup structure of the finite classical groups}, London Mathematical Society Lecture Notes Series \textbf{129}, Cambridge University Press, 1990.


\bibitem{LieSax}M.~W.~Liebeck, J.~Saxl, Minimal degrees of primitive permutation groups, with an application to monodromy groups of covers of Riemann surfaces, \textit{Proc. London Math. Soc. (3)} \textbf{63} (1991), 266--314.


\bibitem{LPS1}M.~W.~Liebeck, C.~E.~Praeger, J.~Saxl, A classification of the maximal subgroups of the finite alternating and symmetric groups, \textit{J. Algebra} \textbf{111} (1987), 365--383.

\bibitem{LPS2}M.~W.~Liebeck, C.~E.~Praeger, J.~Saxl, On the O'Nan-Scott theorem for finite primitive permutation groups, \textit{J. Austr. Math. Soc.} \textbf{44} (1988), 389--396.






\bibitem{Pra}C.~E.~Praeger, The inclusion problem for finite primitive permutation groups, \textit{Proc. London Math. Soc. (3)} \textbf{60} (1990), 68--88.


\bibitem{Seress}\'{A}.~Seress, Primitive groups with no regular orbits on the set of subsets, \textit{Bull. London Math. Soc.} \textbf{29} (1997), 697--704.
\end{document}